\documentclass[11pt]{article}

\usepackage{latexsym,color,xcolor,amsmath,amsthm,amssymb,amscd,amsfonts,stmaryrd}
\usepackage{graphicx,bbm}
\usepackage{cancel, mathtools} 
\usepackage{tikz}
\usetikzlibrary{shapes.geometric, arrows}

\tikzstyle{startstop} = [rectangle, rounded corners, 
minimum width=3cm, 
minimum height=1cm,
text centered, 
draw=black, 
]

\tikzstyle{arrow} = [thick,->,>=stealth]

\newtheorem{conjecture}{Conjecture}[section]
\newtheorem{theorem}[conjecture]{Theorem}

\newtheorem{remark}[conjecture]{Remark}
    \newtheorem{lemma}[conjecture]{Lemma}
\newtheorem{proposition}[conjecture]{Proposition}
\newtheorem{corollary}[conjecture]{Corollary}

\newtheorem{definition}[conjecture]{Definition}

\newcommand\independent{\protect\mathpalette{\protect\independent}{\perp}} 
\def\independent#1#2{\mathrel{\rlap{$#1#2$}\mkern2mu{#1#2}}}

\newcommand{\R}{\mathbb{R}}

\renewcommand{\P}{\mathbb{P}}

\newcommand{\E}{\mathbb{E}}

\newcommand{\vol}{\mathrm{Vol}}

\newcommand{\N}{\mathbb{N}}

\date{\vspace{-5ex}}

\author{}

\date{}

\begin{document}

\title{Concentration inequalities for log-concave sequences}

\author{Arnaud Marsiglietti and James Melbourne}
\maketitle

\begin{abstract}

We investigate quantitative implications of the notion of log-concavity through a probabilistic interpretation. In particular, we derive concentration inequalities, moment and entropy bounds for random variables satisfying a precise degree of log-concavity. Along the way, we recover, improve, and simplify several results existing in the literature. Our approach is based on majorization in the convex order.

\end{abstract}


\section{Introduction}
There has been tremendous recent success demonstrating the log-concavity of important combinatorial and geometric sequences \cite{anari2024log, ardila2023lagrangian, backman2023simplicial, berget2023log, branden2020lorentzian,chan2021log, huh2012milnor,huh2015h,huh2012log,huh2022logarithmic,huh2021correlation}, for background see \cite{branden2015unimodality,huh2022combinatorics,stanley1989log}.  The intention of this article is to investigate the quantitative implications of log-concavity through probabilistic interpretation, and in particular to shed light on the level of concentration implied by varying degrees of log-concavity, in the form of tail bounds, moment comparison inequalities, and the identification of maximum entropy distributions.  A key notion will be that of relative log-concavity.

\begin{definition}
    A function $f \colon \mathbb{R} \to [0,\infty)$ is log-concave when
    $$ f((1-t) x + ty) \geq f^{1-t}(x) f^{t}(y) $$ holds for $x,y \in \mathbb{R}$ and $t \in [0,1]$.  Equivalently $f = e^{-V}$ for convex $V \colon \mathbb{R} \to \mathbb{R} \cup \{ \infty \}$. 
\end{definition}

\begin{definition}\label{def: log-concave function on subset}
    A function $g \colon A \to [0,\infty)$ defined on a subset $A \subseteq \mathbb{R}$ is log-concave  when there exists a log-concave function $f \colon \mathbb{R} \to [0,\infty)$ such that
    \[
        f(x) = g(x),
    \]
    for $x \in A$.
\end{definition}

\noindent
Thus a non-negative sequence $x$ on $\mathbb{N}$ is log-concave when it satisfies
\[
    x_k^2 \geq x_{k-1} x_{k+1}
\]
and has ``no internal zeros'', in the sense that $x_n x_m >0$ implies $x_k > 0$ for $n \leq k \leq m$.  We will only consider sequences that satisfy the no internal zeros property and omit mention of this criteria going forward.

\begin{definition}
    A measure $\nu $ is log-concave with respect to $\mu$ if there exists a Radon-Nikodym derivative $\frac{d \nu}{d \mu}$ that is log-concave. For random variables $X \sim \nu$ and $Z \sim \mu$, we write $X \prec_{lc} Z$ when $\nu$ is log-concave with respect to $\mu$.  In this case we say that $X$ is log-concave with respect to $Z$.
\end{definition}

We will identify a non-negative sequence $x$ with measure $\mu \{k \} = x_k$.  When $\sum_{k} x_k = 1$, $\mu$ is  a probability distribution, and we associate a random variable $X$ with $\mathbb{P}(X = k) \coloneqq x_k$.
Our definitions are slightly more abstract than strictly necessary to handle the discrete setting, but they allow a unified approach to sequences and the continuous setting (as we will explore in Section \ref{sec:continuous}).
Let us consider some examples.

\begin{enumerate}
    \item A sequence on $\mathbb{N} = \{0, 1, 2, \dots\}$ is log-concave, exactly when, considered as a measure, it is relatively log-concave with respect to the distribution of a geometric random variable on $\N$.  Recall $\mu \sim$ Geometric$(p)$, when $\mu \{k\} = (1-p)^k p$ for $k \in \mathbb{N}$ and $p \in (0,1)$. To match previous literature, we will consider at times log-concave sequences on $\N \setminus \{0\}$, for which they are relatively log-concave with respect to the distribution of a geometric random variable on $\N \setminus \{0\}$, defined as $\mu \{k\} = (1-p)^{k-1} p$ for $k \geq 1$ and $p \in (0,1)$.
    
    \item Conventionally a sequence $x$ on $\mathbb{N}$ is ultra log-concave (ULC) when it satisfies
    \[
        x_k^2 \geq \left( 1 + \frac 1 k \right) x_{k-1} x_{k+1}.
    \]
    This corresponds to the case that $x$ is relatively log-concave with respect to the distribution $\mu$ of a Poisson($\lambda$) random variable,  $\mu \{k \} = e^{-\lambda} \frac{\lambda^k}{k!}$.

    \item A sequence $x$ on $[n] \coloneqq \{0,1,\dots, n \}$ is ultra log-concave of order $n$ (ULC($n$)) when it satisfies 
    \[
        x_k^2 \geq \left(1 + \frac{1}{n-k} \right)\left( 1 + \frac 1 k \right) x_{k-1} x_{k+1}.
    \]
    This corresponds to relative log-concavity with respect to the distribution $\mu$ of a binomial$(n,p)$, $\mu \{k \} = \binom{n}{k} p^k (1-p)^{n-k}$.

    \item \label{item: intrinsic volume example} When $x$ is the sequence of intrinsic volumes associated to a convex body in $\mathbb{R}^d$, by the Alexandrov-Fenchel inequality $x$ is relatively log-concave with respect to the intrinsic volume sequence $\mu$ of $d$-dimensional Euclidean ball, given by
    $\mu\{k\} = C_{d,\lambda} \binom{d}{k} \Gamma \left( \frac{d}{2} +1 \right) \lambda^k$ for $k \in [d]$ and $C_{d,\lambda}$ a normalizing constant.
\end{enumerate}

In the continuous setting the implications of log-concavity have been thoroughly explored (see, e.g., \cite{saumard2014log}). However, the perspective of relative log-concavity still gives some strengthenings of previous known results through very simple proofs.

\begin{enumerate}
    \item In the continuous setting a function $f \colon (0,\infty) \to [0,\infty)$ is log-concave when it is log-concave with respect to the distribution of an exponential$(\lambda)$ $\sim \lambda e^{-\lambda x}$.

    \item The log-concave random variables of order $p > 0$, introduced by Bobkov in \cite{bobkov2010gaussian} (see also \cite{bobkov2003spectral, bobkov2011concentration, klartag2007central}), coincide exactly with those that are log-concave with respect to a Gamma($p,\beta$) distribution.

\end{enumerate}

Log-concavity is related to and in some ways parallel to the study of total positivity \cite{Ando1987totally, karlin1968total}, or the study of P\'olya frequency sequences.  A P\'olya frequency sequence $x$ of order $r$, which we abbreviate as $PF_r$, is one such that the matrix $(M_{ij}) = (x_{i - j})$ is totally positive of order $r$, in the sense that any $k \times k$ minor has positive determinant for $k \leq r$. The class $PF_2$ corresponds to log-concave sequences.  When a sequence  belongs to $PF_r$ for every $r$ it is called a P\'olya Frequency sequence $(PF)$. We direct the reader to \cite{brenti1989unimodal,brenti1995combinatorics, pitman1997probabilistic} for further background on the appearance and utility of positivity in combinatorics.

\begin{figure}[ht] \label{fig: roadmap}
\caption{Relations between log-concavity notions}
\centering
\includegraphics[width=.8\linewidth]{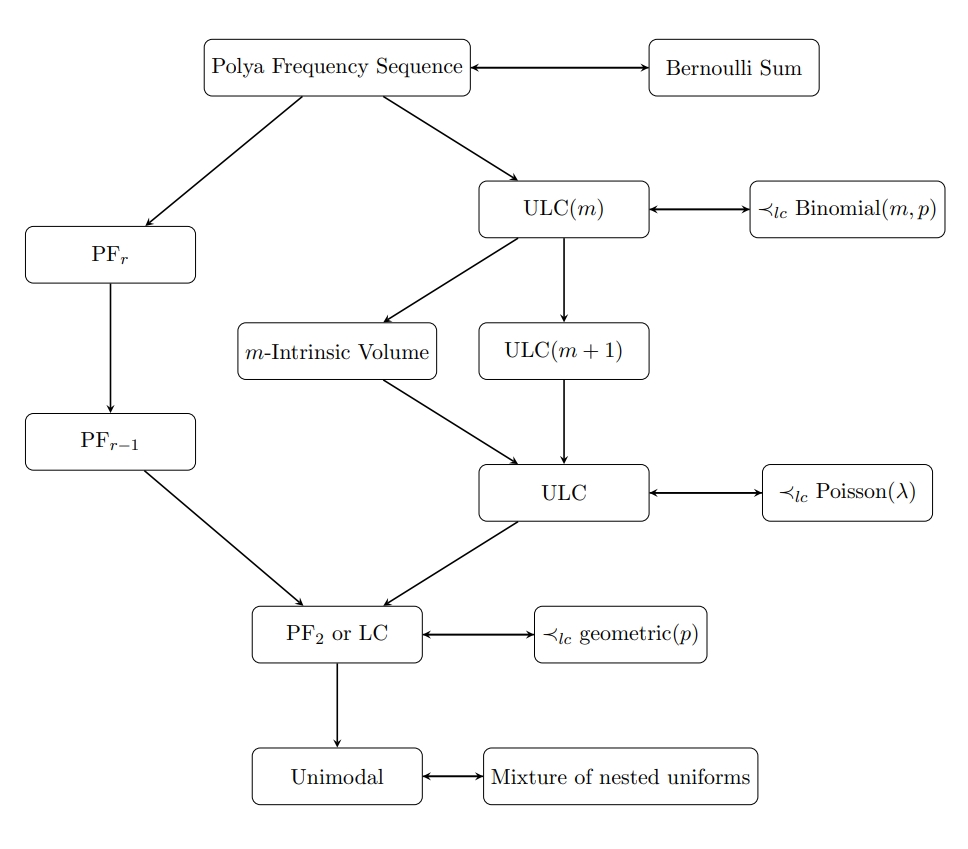}
\end{figure}

 Our work in particular takes inspiration from \cite{pitman1997probabilistic}, where Pitman uses the explicit probabilistic realization (as the distribution of an independent Bernoulli sum) of normalized P\'olya frequency sequences, to derive combinatorial relevant inequalities. The main ingredient is the notion of convex majorization.

\begin{definition}
    A random variable $X$ is majorized by $Z$ in the convex order, written $X \prec_{cx} Z$, when 
    \[
        \mathbb{E}[\varphi(X)] \leq \mathbb{E}[\varphi(Z)]
    \]
    holds for any convex function $\varphi$.
\end{definition}

 Our approach  will be to identify extremal elements in the convex order of a relevant class of variables (those with densities log-concave with respect to a chosen density, typically constrained to satisfy a certain linear constraint), and use direct computations on said extremizers to derive concentration type inequalities. This idea is not new. The observation that relative log-concavity of random variables with matching expectation implies domination in the convex order  was observed explicitly by Whitt \cite{whitt1985uniform}, and is an immediate corollary of (and essentially implicit in) Karlin and Novikoff \cite{karlin1963generalized}.  Moreover, Karlin and Novikoff do not attribute novelty to their ``crossings of density'' arguments in \cite{karlin1963generalized} which they claim are implicit in ``Inequalities'' book of Hardy, Littlewood, and P\'olya originally published in 1934.  For example, see \cite[Exercise 249]{GLP34inequalities}, where a characterization of the convex order is given, that is easily implied from the assumption of a two crossing.

 As a consequence we will deliver many sharp and optimal inequalities for the varying classes of variables. Before detailing these results, let us state that our non-technical message is the following:  Ultra log-concave variables of order $n$ enjoy at least Gaussian concentration phenomena, ultra log-concave variables provide Poisson-type concentration, and intrinsic volumes are subject to concentration at least as strong as the sequences of the Euclidean ball, while log-concave variables give exponential-type concentration.

Through a (to be proven) extension of Whitt \cite{whitt1985uniform} (see Theorem \ref{thm: extension of convex domination theorem}), we will derive new sharp moment comparison results between the expectation and other moments of discrete log-concave random variables.

The article is organized as follows. The background on convex majorization and the proof techniques are presented in Section \ref{background}, where a generalization of Whitt's convex majorization argument is extended to increase the range of application of this technique. Section \ref{sec:lc} demonstrates concentration, moment, and entropy bounds for log-concave probability sequences, and shows that these distributions satisfy exponential-type concentration.  In particular this generalizes and extends the results of Janson \cite{janson2018tail} on the tail bounds of sums of independent geometric random variables. Section \ref{sec: ULC} establishes Poisson-type and Gaussian concentration for ULC and ULC($n$) log-concave probability sequences. For example, thanks to the resolution of the Mason conjecture \cite{anari2024log, branden2020lorentzian}, this gives sharp Gaussian concentration for the size of an independent set in a matroid. In Section \ref{sec: intrinsic}, concentration inequalities for intrinsic volumes are presented, and our approach immediately and easily implies all the results of \cite{aravinda2021concentration} and \cite{lotz2020concentration} as special cases. In particular, we identify the Euclidean ball as a maximizer of the ``intrinsic entropy'' for fixed ``central intrinsic volume'', correcting the claim of \cite{lotz2020concentration} that the cube enjoyed such a distinction. Section \ref{sec:continuous} presents similar results for the ``continuous'' setting, where we extend Janson's results on sums of independent exponential random variables \cite{janson2018tail}, as well as Bobkov's concentration results \cite{bobkov2010gaussian} for distributions that are log-concave of order $p$.

\section{Background on Convex Majorization}\label{background}

\subsection{Majorization of Log-Concave Distributions} \label{sec: Majorization of logconcave}

For subsets $A,B \subseteq \mathbb{R}$ we introduce the notation $A \prec B$ when $\sup A \leq \inf B$.

\begin{definition}\label{minimal-part}
    For a set $S \subseteq \mathbb{R}$, a function $F \colon S \to \mathbb{R}$ has no more than $k$-zero crossings if there exists a {\it sign preserving ordered partition} of $S$ of size $k+1$, that is, $\{S_i \}_{i=0}^k$, with $S_0 \prec \cdots \prec S_k$ and $\bigcup_{i=0}^n S_i = S$, such that 
    \begin{align} \label{eq: partition equation}
          F(x) F(y) \geq 0, \ \ \hbox{ for $x, y \in S_i$.}
    \end{align}
    If $F$ has no more than $k$-zero crossings for some finite $k$, we say that it has $n$-zero crossings if $n$ is minimum over all $k$, such that $F$ has no more than $k$-zero crossings. In this case, $\{S_i \}_{i=0}^n$ satisfying \eqref{eq: partition equation} is called a minimal partition.
\end{definition}

The following theorem describes a simple sufficient condition for convex majorization. Intuitively, it says that majorization holds for random variables of matched mean, when their density functions cross twice.

\begin{theorem} \label{thm: Karlin type theorem basic}
    Let $\mu$ be a measure on $[0, +\infty)$. Suppose that $f$ and $g$ are Borel measurable functions on $[0, +\infty)$ such that
    \begin{align*}
        \int_0^\infty f(x) d\mu(x) = \int_0^\infty g(x) d\mu(x) 
    \end{align*}
    and 
    \begin{align*}
        \int_0^\infty x f(x) d\mu(x) = \int_0^\infty x g(x) d\mu(x)
    \end{align*}
    and the integrals are finite. If there exists an interval $I \coloneqq [a,b] \subseteq (0,\infty)$ such that $g(y) \leq f(y)$ for $y \in I$ while $g(y) \geq f(y)$ for $y \notin I$, then $\varphi \colon (0,\infty) \to \mathbb{R}$ convex implies,
    \begin{align*}
        \int_0^\infty \varphi(x) f(x) d\mu(x) \leq \int_0^\infty \varphi(x) g(x) d\mu(x).
    \end{align*}
\end{theorem}

\begin{proof}
    Taking $\tilde{\varphi}(x) = \varphi(x) - (mx+k)$ with $m,k$ chosen such that $\tilde{\varphi}(a) = \tilde{\varphi}(b) = 0$ (explicitly $m \coloneqq \frac{\varphi(b)- \varphi(a)}{b-a}$ and $k \coloneqq \frac{\varphi(a)b - \varphi(b)a}{b-a}$), then
    \[
        \int [g -f] \varphi \ d\mu = \int [g-f] \tilde{\varphi} \ d\mu \geq 0.
    \]
    The above equality holds since $\tilde{\varphi} - \varphi$ is an affine function and hence $f d\mu$ and $g d\mu$ have identical integrals.  The inequality follows since $(g-f) \tilde{\varphi}$ is a non-negative function, by assumptions on $f$ and $g$, and by construction.
\end{proof}

The next theorem describes a large class of random variables satisfying convex majorization.

\begin{theorem}\label{Major-log}

Let $\mu$ be a measure on $[0,+\infty)$. For $X$ a non-negative random variable that is log-concave with respect to $\mu$, and $Z$ a non-negative random variable that is log-affine with respect to $\mu$ on the entire support of $\mu$, and satisfying $\mathbb{E}[Z] = \mathbb{E}[X]$, we have
\begin{align*}
    X \prec_{cx} Z.
\end{align*}

\end{theorem}

\begin{proof}
    We will show that the probability density function of $X$ and $Z$ with respect to $\mu$ have exactly 2 crossings and apply Theorem \ref{thm: Karlin type theorem basic}. Let us denote by $f$ the p.d.f. of $X$ and by $a$ the p.d.f. of $Z$. Since $f$ is log-concave, and $a$ is log-affine, $f$ and $a$ have no more than two crossings.
    
    If $f = a$, there is nothing to prove, so let us assume that there exists $x$ such that $f(x) \neq a(x)$. In this case, we claim $f$ and $a$ must have exactly two crossings.  If there are no crossings then we have $f(x) \leq a(x)$ for all $x$ or $f(x) \geq a(x)$ for all $x$ with strict inequality for some $y$.  In either case, this contradicts $\int f(x) d\mu(x) = \int a(x) d\mu(x) = 1$.
    
    To have exactly one crossing, would contradict $\mathbb{E}[X] = \mathbb{E}[Z]$.  Indeed, say $a(x) \geq f(x)$ for $x \leq x_1$ and $a(x) \leq f(x)$ for $x > x_1$. But this would imply $\mathbb{P}(Z > t) \leq \mathbb{P}(X > t)$ for all $t >0$, with a strict inequality for some $t$ (else $X$ and $Z$ would be the same distribution) and hence
    \begin{align*}
        \mathbb{E}[Z] = \int_0^\infty \mathbb{P}(Z > t) dt  < \int_0^\infty \mathbb{P}(X > t ) dt = \mathbb{E}[X] 
    \end{align*}
    gives the contradiction. Thus applying Theorem \ref{thm: Karlin type theorem basic}, we obtain $X \prec_{cx} Z$.
\end{proof}

If $\mu$ and $\nu$ are measures on a measurable space $(E,\mathcal{F})$, we consider the pushforward measures $T_*\mu$ and $T_*\nu$ as measures on the measurable space induced by $T$, $(T(E), T(\mathcal{F}))$, where the $\sigma$-algebra $T(\mathcal{F})$ is defined by $A \subseteq T(E)$ belongs to $T(\mathcal{F})$ if and only if $T^{-1}(A) \in \mathcal{F}$. 

\begin{theorem} \label{thm: density pushforward}
Let $\nu$ and $\mu$ be measures on $\mathbb{R}$ such that $\frac{d\nu}{d\mu}$ exists and let $T$ be a non-decreasing function. Then, on $T(supp(\mu))$, $T_* \nu$ has a density with respect to $T_*\mu$ which can be expressed by
\begin{align*}
    \frac{dT_*\nu}{dT_*\mu}(y) := f^{*}(y) = \begin{cases}
        \frac{\nu(T^{-1} (\{y \}))}{\mu(T^{-1} (\{y \}))}  &\text{ for } \# \{ T^{-1} (\{y \})\} > 1
            \\
    \frac{d \nu}{ d\mu} (x) &\text{ for } \# \{ T^{-1} (\{y\}) \} = 1,
        \end{cases}
\end{align*}
where $ x$ is the unique point such that $T(x) = y$ and we use the convention that $\frac{0}{0} \coloneqq 0.$
\end{theorem}

\begin{proof}
Since $T$ is non-decreasing, we have that for $y \in T(\R)$, $T^{-1}(\{y\}) = \{x \in \R : T(x) = y\}$ is an interval (possibly reduced to a singleton). Therefore, there are only countably many $y$'s such that $\# T^{-1}(\{y\}) > 1$.
Let us enumerate such $y$'s as $\{y_i\}_{i \in I}$, for some index set $I \subseteq \N$, and let us denote $K = \{y \in T(\R) : \# T^{-1}(\{y\}) = 1\}$. Note that for all $i \in I$,
\begin{eqnarray*}
T_*\nu(\{y_i\}) = \nu(\{T^{-1}(\{y_i\})\}) & = & \frac{\nu(\{T^{-1}(\{y_i\})\})}{\mu(\{T^{-1}(\{y_i\})\})} \mu(\{T^{-1}(\{y_i\})\}) \\ & = & f^*(y_i) \mu(\{T^{-1}(\{y_i\})\}) \\ & = & \int_{\{y_i\}} f^* dT_*\mu.
\end{eqnarray*}
Therefore, for any Borel set $A \subset \R$,
$$ T_*\nu(A \cap K^c) = \sum_{i \in I : y_i \in A} T_*\nu(\{y_i\}) = \sum_{i \in I : y_i \in A} \int_{\{y_i\}} f^* dT_*\mu = \int_{A \cap K^c} f^* dT_*\mu. $$
On the other hand, since $T^{-1}$ defines an injective map on $K$, one may write for $y \in K$, $T^{-1}(\{y\}) = T^{-1}(y) \in \R$, so that if $x \in T^{-1}(K)$, we have $T^{-1}(T(x)) = x$, and thus
$$ \frac{d \nu}{d \mu}(x) = \frac{d \nu}{d \mu}(T^{-1}(T(x))) = f^*(T(x)). $$
Therefore, for any Borel set $A \subset \R$,
\begin{eqnarray*}
T_*\nu(A \cap K) = \nu(T^{-1}(A \cap K)) & = & \int 1_{A \cap K}(T(x)) \frac{d \nu}{d \mu}(x) d\mu(x) \\ & = & \int 1_{A \cap K}(T(x)) f^*(T(x)) d\mu(x) \\ & = & \int_{A \cap K} f^* dT_*\mu.
\end{eqnarray*}
We conclude by writing
$$ T_*\nu(A) = T_*\nu(A \cap K) + T_*\nu(A \cap K^c) = \int_{A \cap K} f^* dT_*\mu + \int_{A \cap K^c} f^* dT_*\mu = \int_{A} f^* dT_*\mu. $$
\end{proof}

\begin{theorem}\label{cross-increase}
    Suppose that $\nu$ and $\gamma$ have densities $f$ and $g$ on a set $S \subseteq \mathbb{R}$ with respect to $\mu$, and that $f - g$ has no more than $n$-zero crossings, and that $T$ is a non-decreasing function. Then $T_* \nu$ and $T_* \gamma$ have densities $f^*$ and $g^*$ on $T(S)$ with respect to $T_*\mu$, and $f^* - g^*$ has no more than $n$-zero crossings.
\end{theorem}

\begin{proof}
    Using the densities $f^*$ and $g^*$ supplied by Theorem \ref{thm: density pushforward}, 
    \begin{align} \label{eq: pushforward density}
        F^*(y) \coloneqq f^* - g^* (y) =  \begin{cases}
        \frac{\nu(T^{-1} (\{y \})) - \gamma(T^{-1} (\{y \}))}{\mu(T^{-1} (\{y \}))}  &\text{ for } \# \{ T^{-1} (\{y \}) \} > 1
            \\
    (f-g) (x) &\text{ for } \# \{ T^{-1} (\{y\}) \} = 1
    \end{cases},
    \end{align}
    where again, $x$ is the unique point such that $T(x) = y$.
Let $\{S_i\}_{i=0}^n$ denote a sign preserving ordered partition of $S$ with respect to $F:=f-g$, see Definition \ref{minimal-part}, so that for all $x,y \in S_i$, $F(x)F(y) \geq 0$.  Define a partition of $T(S)$ inductively by $J_0 = T(S_0)$,
and for $1 \leq i \leq n$,
\begin{align*}
    J_i = T(S_i) \setminus \bigcup_{k=0}^{i-1} T \left( S_k \right).
\end{align*}
The $J_i$ are by definition a partition of $T(S)$ and satisfy $J_i \prec J_{i+1}$ by the monotonicity of $T$.  It remains to check that the partition respects the changes of sign.  If $y_1, y_2 \in J_i$ and $\# \{T^{-1} (\{y_k\}) \} = 1$, then there exist unique $x_k \in S_i$ such that $T(x_k) = y_k$.  Thus,
\[
    F^*(y_1) F^*(y_2) = F(x_1)F(x_2) \geq 0.
\]
If only one of the $y_k$, say $y_1$, satisfies $\#\{T^{-1}(\{y_1\})\} = 1$ then
\[
    F^*(y_1) F^*(y_2) = \frac{\int_{T^{-1}(\{y_2\})} F(x_1)F(x) d\mu(x)}{\mu(T^{-1}(\{y_2\}))} \geq 0,
\]
since $x, x_1 \in S_i$.  Finally if neither $y_k$ is the image of a unique point,
\[
    F^*(y_1) F^*(y_2) = \frac{\int_{T^{-1}(\{y_1\})} \int_{T^{-1}(\{y_2\})} F(x) F(x') \, d\mu(x) d\mu(x')}{\mu(T^{-1}(\{y_1\})) \mu(T^{-1}(\{y_2\}))} \geq 0,
\]
since $x, x' \in S_i$.
\end{proof}

The following generalization of Theorem \ref{Major-log} is needed in some applications, such as moments comparison (see Section \ref{sec: moment bounds}).

\begin{theorem}\label{thm: extension of convex domination theorem}

Let $\mu$ be a measure on $[0,+\infty)$. Suppose that $X$ is log-concave with respect to $\mu$ and that $Z$ is log-affine with respect to $\mu$ on the entirety of the support of $\mu$. If $T$ is a non-decreasing function such that $\mathbb{E}[T(X)] = \mathbb{E}[T(Z)]$, then
$$ T(X) \prec_{cx} T(Z). $$

\end{theorem}

\begin{proof}
    If $X$ is log-concave and $Z$ is log-affine with respect to a reference measure $\mu$, then their densities have at most two crossings. Since $T$ is non-decreasing, the densities of $T(X)$ and $T(Z)$ have at most two crossings by Theorem \ref{cross-increase}. We can then repeat the proof of Theorem \ref{Major-log}. If $T(X) = T(Z)$ the proof is trivial, hence we may assume there is at least one crossing of $T(X)$ and $T(Z)$, and since exactly one crossing would again contradict $\mathbb{E}[T(X)] = \mathbb{E}[T(Z)]$ the proof is complete.
\end{proof}

\subsection{Quantitative implications}\label{sec: applications}

\subsubsection{Concentration inequalities}\label{sec: concentration}

The first application of convex majorization is used toward deriving concentration inequalities. Let $X$ be a log-concave random variable with respect to a reference measure $\mu$. According to Theorem \ref{Major-log}, if $Z$ is $\mu$-log-affine supported on the whole $\{\mu > 0\}$ such that $\E[Z] = \E[X]$, then $X \prec_{cx} Z$. As the result, for all convex function $\varphi$,
\begin{equation}\label{convex}
\E[\varphi(X)] \leq \E[\varphi(Z)].
\end{equation}

The following result demonstrates that Chernoff-type tail bounds on a random variable can be transferred to a random variable it majorizes.  
To this end we define for a random variable $X$, 
\begin{align*}
    \Lambda_X(\lambda) \coloneqq \log \mathbb{E}[ e^{\lambda X}].
\end{align*}
For a real-valued function $f$ defined on an interval $I$, we denote by $f^*$ the Legendre transform, 
\begin{align}
    f^*(t) = \sup_{\lambda \in I} \{ \lambda t - f(\lambda) \},
\end{align}
defined on $I^*$ the set of $t$ such that the supremum is finite.

\begin{theorem} \label{thm: Chernoff majorization}
For $X \prec_{cx} Z$, 
\begin{align*}
    \mathbb{P}( X \geq t) \leq \exp [ \ - \Lambda_+^*(t)], \hspace{8mm} \mathbb{P}( X \leq t) \leq \exp [ \ - {\Lambda}_-^*(t)],
\end{align*}
where $\Lambda_+$ is the function $\Lambda_Z(\lambda) = \log \mathbb{E}[e^{\lambda Z}]$ restricted to $\lambda >0$ while $\Lambda_-$ is the restriction of $\Lambda_Z$ to $\lambda < 0$.
\end{theorem}

\begin{proof}
    For $\lambda > 0$, the standard approach through Markov's inequality gives
    \begin{align*}
        \mathbb{P}( X \geq t)
            =
                \mathbb{P}( e^{\lambda X} \geq e^{\lambda t})
            \leq 
                e^{ \Lambda_X(\lambda) - \lambda t}.
    \end{align*}
     Since $\Lambda_X(t) \leq \Lambda_Z(t)$ follows from $X \prec_{cx} Z$, taking the infimum over $\lambda >0$ yields the result.  Similarly, for $\lambda < 0$,
    \begin{align*}
        \mathbb{P}( X \leq t)
            =
                \mathbb{P}( e^{\lambda X} \geq e^{\lambda t})
            \leq 
                e^{ \Lambda_X(\lambda) - \lambda t}.
    \end{align*}
   Taking the infimum over $\lambda <0$ completes the proof. 
\end{proof}

The above method is sometimes referred to as the Cram\'er-Chernoff method in the literature (see, e.g., \cite{boucheron}).



\subsubsection{Moment Bounds} \label{sec: moment bounds}

The next application of majorization is used to provide comparisons between moments.

\begin{theorem}\label{mom-compar}
Let $0 < \alpha < \beta$. For a random variable $X$ log-concave with respect to $Z$, and satisfying $\mathbb{E}[Z^{\alpha}] = \mathbb{E}[X^{\alpha}]$, we have
\begin{align}
    \mathbb{E}[X^{\beta}]^{\frac 1 \beta} \leq A_{\alpha,\beta} \mathbb{E}[X^{\alpha}]^{\frac 1 \alpha}
\end{align}
where
\begin{align*}
    A_{\alpha,\beta} = \frac{\mathbb{E}[Z^{\beta}]^{\frac 1 \beta}}{\mathbb{E}[Z^{\alpha}]^{\frac 1 \alpha}}.
\end{align*}
\end{theorem}

\begin{proof}
    The result is obtained by taking $T(x) = x^{\alpha}$, the convex function $\varphi(x) = x^{\frac \beta \alpha}$, and $\mu$ the distribution of $Z$ in Theorem \ref{thm: extension of convex domination theorem}.
\end{proof}

\subsubsection{Maximum Entropy Distributions} \label{sec: Maximum entropy}

Lastly, convex majorization can be used to provide entropy bounds. Let us recall the definition of the R\'enyi entropy.
\begin{definition}[R\'enyi Entropy]
    For a random variable $X$ with probability mass function $x_k := \mathbb{P}(X = k)$, and $\alpha \in (0,1) \cup (1,\infty)$,
    \[
        H_\alpha(X) \coloneqq \frac{1}{1-\alpha} \log \left( \sum_k x_k^\alpha \right).
    \]
    Further, $H_0(X) \coloneqq \log |\{x_n >0\}|$ with $| \cdot |$ denoting cardinality, $H_1(X) = H(X)$ is the usual Shannon entropy  $ -\sum_k x_k \log x_k,$ and $H_\infty(X) = - \log \|x \|_\infty$, with $\|x\|_\infty \coloneqq \max_k x_k$.  When $X$ is a continuous random variable with density function $f$ with respect to Lebesgue measure, we write 
    $$
        h_\alpha(X) \coloneqq \frac{1}{1-\alpha} \log \left( \int f^\alpha(x) dx \right),
    $$
    with $h_0(X) \coloneqq \log |\{f >0\}|$ where $| \cdot |$ denotes volume, $h_1(X) = h(X) = -\int f(x) \log f(x)dx$, and $h_\infty(X) = - \log \|f \|_\infty$, with $\|f\|_\infty$ denoting the essential supremum of $f$ with respect to Lebesgue measure.
\end{definition}

An application of Theorem \ref{Major-log} together with the following lemma easily yield R\'enyi entropy maximization within subclasses of log-concave distributions.

\begin{lemma}\label{entropy}

Let $X \sim f$, $Z \sim g$ be random variables where $f,g$ are densities with respect to the geometric distribution in the discrete case, or with respect to the exponential distribution in the continuous case. In order to prove $H_{\alpha}(X) \leq H_{\alpha}(Z)$ or $h_{\alpha}(X) \leq h_{\alpha}(Z)$, it suffices to prove
\begin{eqnarray*}
\E[g^{\alpha-1}(X)] & \leq & \E[g^{\alpha-1}(Z)], \quad {\text{ if }} \, \alpha \in (0,1), \\
-\E[\log(g(X))] & \leq & -\E[\log(g(Z))], \quad {\text{ if }} \, \alpha = 1,
    \\
\E[g^{\alpha-1}(X)] & \geq & \E[g^{\alpha-1}(Z)], \quad {\text{ if }} \, \alpha \in (1,\infty).
\end{eqnarray*}

\end{lemma}

The proof is an application of H\"older's inequality, and the non-negativity of the relative entropy. A proof can be found in detail in \cite[Lemma 3.25]{MMX17:1}.

\begin{theorem} \label{thm: maximum entropy dist}
    Let $\alpha \leq 1$. Let $Z$ be a discrete or continuous log-concave random variable. If $X$ is such that $X \prec_{lc} Z$ and $\mathbb{E}[X] = \mathbb{E}[Z]$, then
    \begin{align*}
        H_\alpha(X) \leq H_\alpha(Z) \quad \mbox{or} \quad  h_\alpha(X) \leq h_\alpha(Z).
    \end{align*}
\end{theorem}

\begin{proof}
In the discrete setting, denote by $g$ the probability mass function of $Z$. Consider $\widetilde{g}$ the piecewise linear extension of $g$ so that $\widetilde{g}$ is a log-concave function on $[0,+\infty)$ (see, e.g., \cite{BMM20} Proposition 5.1). Therefore the functions $\varphi(x) = \widetilde{g}^{\alpha-1}(x)$ when $\alpha \in (0,1)$ and $\psi(x) = -\log(\widetilde{g}(x))$ are convex. It remains to apply Theorem \ref{Major-log} together with Lemma \ref{entropy}. The proof in the continuous setting is similar and more straightforward.
\end{proof}

As a consequence, with $\alpha \leq 1$, the geometric distribution and exponential distribution have maximum $\alpha$-entropy among non-negative discrete and continuous log-concave distributions with fixed expectation, extending the classical fact that the geometric and exponential distributions have maximum Shannon entropy among all positive distributions of fixed expectation. Further the Poisson distribution has maximum $\alpha$-entropy for fixed expectation extending \cite{Yu09:2}, and the binomial$(p,n)$ has maximum $\alpha$-entropy among ULC$(n)$ variables with expectation $pn$, extending the result in the Shannon case given by Yu \cite{yu2008maximum} proven through ``thinning techniques'', which extended Harremo\"es \cite{Har01} who had proven the same result for the subset of ULC$(n)$ consisting of independent Bernoulli sums of length $n$.

\section{Log-concave Sequences}\label{sec:lc}

As discussed in the introduction, log-concavity of a non-negative sequence $\{x_k\}$ can be written as the inequality 
\begin{equation} \label{eq: vanilla log-concave sequence}
    x_k^2 \geq x_{k-1} x_{k+1}
\end{equation}
with the additional assumption that the sequence has no internal zeros, in the sense that $x_n x_m > 0$ implies $x_k > 0$ for $ n < k < m$.  Constructing $f = e^{- V}$ with $V$  the linear interpolation of the points $(k,- \log x_k)$ yields $V$ convex and hence $x$ being log-concave as a sequence agrees with that of Defintion \ref{def: log-concave function on subset} (see, e.g., \cite[Proposition 5.1]{BMM20}). A useful and thematic observation is that for a non-negative integer-valued random variable $X$, the sequence $x_k \coloneqq \mathbb{P}(X = k)$ is log-concave if and only if $X \prec_{lc} Z$ where $Z$ is a geometric random variable (on the appropriate support). Typically we will consider a sequence to be defined on $\mathbb{N} = \{0,1,2, \dots \}$ or $\mathbb{N} \setminus \{0 \}$, however in the context of the translation invariant inequality \eqref{eq: vanilla log-concave sequence}, other subspaces of the integers can be used without complication depending on the application.

In the next theorem, to match previous literature, we consider log-concave random variables on $\mathbb{N} \setminus \{0\}$. Recall that $Z$ is  geometric$(p)$ on $\mathbb{N} \setminus \{0\}$ if there exists $p \in (0,1)$ such that $\mathbb{P}(Z = k ) = (1-p)^{k-1} p$ for $k \geq 1$. Our first main result extends the tail bounds for sums of independent geometric random variables due to Janson \cite{janson2018tail}, to sums of positive log-concave random variables.

\begin{theorem}\label{thm: conc-geom-sum}

Let $X_1, \dots, X_n$, $n \geq 1$, be independent discrete log-concave random variables on $\N \setminus \{0\}$. Denote $S_n = \sum_{i=1}^n X_i$. Then, for all $t \geq 1$,
$$ \P(S_n \geq t \E[S_n]) \leq e^{-\frac{\E[S_n]}{\max_i \E[X_i]}(t-1-\log(t))}, $$
and for all $t \leq 1$,
$$ \P(S_n \leq t \E[S_n]) \leq e^{-\frac{\E[S_n]}{\max_i \E[X_i]}(t-1-\log(t))}. $$

\end{theorem}

Taking $X_i$ to be geometric random variables on $\mathbb{N} \setminus \{0\}$ recovers \cite[Theorem 2.1]{janson2018tail} and \cite[Theorem 3.1]{janson2018tail}.


To establish Theorem \ref{thm: conc-geom-sum}, we will have use for the following elementary lemma. Recall that an exponential$(p)$ random variable has density function $f(x) = p e^{-px}$ for $x \geq 0$ and $p > 0$.

\begin{lemma}\label{lem: moment generating function bound}
    For $G$ a geometric$(p)$ random variable on $\N \setminus \{0\}$ and $W$ an exponential$(p)$ random variable, and $\theta< \lambda \leq p$,
    \begin{align*}
        \mathbb{E}[e^{\theta G}] \leq \mathbb{E}[e^{\theta W}] = \frac{p}{p - \theta} \leq e^{- \frac \lambda p \log \left[ 1 - \frac \theta \lambda \right]}.
    \end{align*}
\end{lemma}

\begin{proof}
    When $W$ is exponential$(p)$,
    \begin{align*}
        \mathbb{E}[e^{\theta W}] = \frac{p}{p-\theta } = e^{- \log \left[ 1- \frac \theta p\right]} \leq e^{- \frac \lambda p \log \left[ 1 - \frac \theta \lambda \right]}.
    \end{align*}
    When $G$ is geometric$(p)$,
    \begin{align*}
        \mathbb{E}[e^{\theta G}] = \frac{p}{p - (1-e^{-\theta })} \leq \frac{ p}{p-\theta } = \mathbb{E}[e^{\theta W}],
    \end{align*}
    which gives the result.
\end{proof}

\begin{proof}[Proof of Theorem \ref{thm: conc-geom-sum}]
Let $Z_i$, $i = 1, \dots, n$, be independent geometric distributions on $\N \setminus \{0\}$ with parameter $p_i = \frac{1}{\E[X_i]}$, chosen so that $\mathbb{E}[X_i] = \mathbb{E}[Z_i]$. Thus, by Theorem \ref{Major-log}, $X_i \prec_{cx} Z_i$. Hence, the result follows from estimates of the moment generating function of $\sum_{i=1}^n Z_i$.

For all $0 < \theta < \lambda \coloneqq \min_{i} p_i = (\max_i \mathbb{E}[X_i])^{-1}$, applying the product structure of the moment generating function, $X_i \prec_{cx} Z_i$ applied to the convex function $x \mapsto e^{\theta x}$, and then Lemma \ref{lem:  moment generating function bound}, 
\begin{eqnarray*} 
    \E[e^{\theta S_n}] 
        \leq 
            \Pi_{i=1}^n \E[e^{\theta Z_i}] 
        \leq 
            e^{- \sum_{i=1}^n \frac{\lambda}{p_i} \log \left(1 - \frac{\theta}{\lambda} \right)}     
        = 
            e^{- \lambda \E[S_n] \log \left(1 - \frac{\theta}{\lambda} \right)}.
\end{eqnarray*}
Therefore, by Markov's inequality
$$ \P(S_n \geq t \E[S_n]) \leq e^{-\theta t \E[S_n] - \lambda \E[S_n] \log(1 - \frac{\theta}{\lambda})}. $$
Taking $\theta = \lambda (1 - \frac{1}{t})$ yields the result for the large deviation bound. A similar argument applied to negative $\theta$ yields the small deviation bound.
\end{proof}

Taking $n=1$ in Theorem \ref{thm: conc-geom-sum} gives the following corollary.

\begin{corollary}\label{conc-geom}

Let $X$ be a discrete log-concave random variable on $\N \setminus \{0\}$. Then, for all $t \geq 1$,
$$ \P(X \geq t \E[X]) \leq t e^{1-t}, $$
and for all $t \leq 1$,
$$ \P(X \leq t \E[X]) \leq te^{1-t}. $$

\end{corollary}

\begin{remark}\label{general-bound-discrete}

One may obtain concentration bounds for $\sum_{i=1}^n a_i X_i$, $a_i > 0$, via the same majorization approach since if $X_i \prec_{cx} Z_i$, then for all $\theta \in \R$,
$$ \E[e^{\theta a_i X_i}] \leq \E[e^{\theta a_i Z_i}]. $$
For example, for log-concave distributions on $\N \setminus \{0\}$, following the proof of Theorem \ref{thm: conc-geom-sum} yields
$$ \P \left( \sum_{i=1}^n a_i X_i \geq t \, \E[\sum_{i=1}^n a_i X_i] \right) \leq e^{- \left(\min_i a_i \E[X_i] \right) \E[\sum_{i=1}^n a_i X_i](t-1-\log(t))}, $$
and for all $t \leq 1$,
$$ \P \left( \sum_{i=1}^n a_i X_i \leq t \, \E[\sum_{i=1}^n a_i X_i] \right) \leq e^{- \left(\min_i a_i \E[X_i] \right) \E[\sum_{i=1}^n a_i X_i](t-1-\log(t))}. $$

\end{remark}

\begin{remark}

One may obtain similar concentration bounds for log-concave distributions on $\N$ by applying the majorization argument with the geometric distribution on $\N$, $\P(Z = k) = p(1-p)^k$, $k \geq 0$.

\end{remark}

We now move on to discuss moment bounds for discrete log-concave distributions on $\N$. To establish such bounds, by Theorem \ref{mom-compar} we need to estimate the quantity $A_{\alpha,\beta} = \mathbb{E}[Z^{\beta}]^{\frac 1 \beta} / \mathbb{E}[Z^{\alpha}]^{\frac 1 \alpha}$, with $\beta \geq \alpha$, where $Z$ is a geometric distribution on $\N$. First, we note that there is no absolute comparison for all $Z$ geometric. For example, taking $\beta=2$ and $\alpha=1$, we have for $Z \sim (1-p)^k p$, $k \geq 0$,
$$ \E[Z] = \frac{1-p}{p}, \quad  \E[Z^2] = \frac{(1-p)(2-p)}{p^2}. $$
Therefore,
$$ \frac{\E[Z^2]^{1/2}}{\E[Z]} = \frac{\sqrt{(1-p)(2-p)}}{p} \frac{p}{1-p} = \frac{\sqrt{2-p}}{\sqrt{1-p}} \xrightarrow[p \to 1]{} +\infty. $$
Hence, the supremum over all $Z$ geometric on $\N$ is $+\infty$. Nonetheless, we can establish the following moment bound. 

\begin{proposition}\label{moment-discrete}

Let $X$ be a discrete log-concave random variable on $\N$. Then, for all $0 < \alpha \leq \beta$, we have
$$ \E[X^{\beta}]^{\frac 1 \beta} \leq \frac{\Gamma(\beta +1)^{1/\beta}}{\Gamma(\alpha+1)^{1/\alpha}} \E[X^{\alpha}]^{\frac 1 \alpha} e^{\frac{\Gamma(\alpha+1)^{1/\alpha}}{\alpha \E[X^{\alpha}]^{1/\alpha}}}. $$
In particular, if $\E[X^{\alpha}]^{\frac 1 \alpha} \geq 1$, we have
$$ \E[X^{\beta}]^{\frac 1 \beta} \leq c_{\alpha,\beta} \E[X^{\alpha}]^{\frac 1 \alpha}, $$
where $c_{\alpha, \beta}$ depends only on $\alpha$ and $\beta$. Moreover, one may take $c_{\alpha, \beta} = e^{\frac{\Gamma(\alpha+1)^{1/\alpha}}{\alpha}} \frac{\Gamma(\beta +1)^{1/\beta}}{\Gamma(\alpha+1)^{1/\alpha}}$.

\end{proposition}

\begin{proof}
Note that for $Z$ geometric on $\N$,
$$ \E[Z^{\alpha}] = p \sum_{k \in \N} k^{\alpha} (1-p)^k = p \sum_{k \in \N} k^{\alpha} e^{-k \log(1/(1-p))}. $$
Also, we have
$$ \int_0^{+\infty} x^{\alpha} e^{-x \log(1/(1-p))} dx = \frac{\Gamma(\alpha+1)}{\log(1/(1-p))^{\alpha+1}}. $$
On the other hand,
\begin{eqnarray*}
\int_0^{+\infty} x^{\alpha} e^{-x \log(1/(1-p))} dx & = & \sum_{k \in \N} \int_k^{k+1} x^{\alpha} e^{-x \log(1/(1-p))} dx \\ & \leq & \sum_{k \in \N} e^{-k \log(1/(1-p))} \int_k^{k+1}  x^{\alpha} dx
\\ & = & \frac{1}{\alpha+1} \left[\sum_{k \in \N} (k+1)^{\alpha+1} (1-p)^k  - \sum_{k \in \N} k^{\alpha+1} (1-p)^k \right]
\\ & = & \frac{1}{\alpha+1} \left[ \frac{\E[Z^{\alpha + 1}]}{p(1-p)} - \frac{\E[Z^{\alpha + 1}]}{p} \right]
\\ & = & \frac{1}{\alpha+1} \frac{1}{1-p} \E[Z^{\alpha+1}].
\end{eqnarray*}
Therefore,
$$ \E[Z^{\alpha+1}] \geq \Gamma(\alpha+2) \frac{1-p}{\log(1/(1-p))^{\alpha+1}}. $$
Similarly, we have
\begin{eqnarray*}
\int_0^{+\infty} x^{\alpha} e^{-x \log(1/(1-p))} dx & \geq & \sum_{k \in \N} e^{-(k+1) \log(1/(1-p))} \int_k^{k+1}  x^{\alpha} dx \\ & = & \frac{1}{\alpha+1} \left[\sum_{k \in \N} (k+1)^{\alpha+1} (1-p)^{k+1}  - \sum_{k \in \N} k^{\alpha+1} (1-p)^{k+1} \right]
\\ & = & \frac{1}{\alpha+1} \E[Z^{\alpha+1}].
\end{eqnarray*}
Finally, we have the comparison
$$ (1-p)^{\frac{1}{\alpha+1}} \Gamma(\alpha+2)^{\frac{1}{\alpha+1}} \leq \log(1/(1-p)) \E[Z^{\alpha+1}]^{\frac{1}{\alpha+1}} \leq \Gamma(\alpha+2)^{\frac{1}{\alpha+1}}. $$
We deduce that for all $0 < \alpha \leq \beta$,
\begin{equation}\label{A-alpha}
1 \leq A_{\alpha,\beta} = \frac{\mathbb{E} [Z^{\beta}]^{\frac 1 \beta}}{\mathbb{E} [Z^{\alpha}]^{\frac 1 \alpha}} \leq \frac{1}{(1-p)^{\frac{1}{\alpha}}}\frac{\Gamma(\beta +1)^{1/\beta}}{\Gamma(\alpha+1)^{1/\alpha}}.
\end{equation}
Now, let $X$ be a discrete log-concave random variable on $\N$ and assume that $\beta \geq \alpha > 0$. Then, the constraint $\E[X^{\alpha}] = \E[Z^{\alpha}]$ implies
$$ \E[X^{\alpha}]^{1/\alpha} = \E[Z^{\alpha}]^{1/\alpha} \leq \frac{\Gamma(\alpha+1)^{\frac{1}{\alpha}}}{\log(1/(1-p))}. $$
Therefore,
\begin{equation}\label{1/p}
\frac{1}{1-p} \leq e^{\frac{\Gamma(\alpha+1)^{\frac{1}{\alpha}}}{\E[X^{\alpha}]^{1/\alpha}}}.
\end{equation}
From convex majorization (Theorem \ref{mom-compar}), we also have that the constraint $\E[X^{\alpha}] = \E[Z^{\alpha}]$ implies $\E[X^{\beta}] \leq \E[Z^{\beta}]$. We finally deduce from \eqref{A-alpha} and \eqref{1/p} that
\begin{eqnarray} \E[X^{\beta}]^{\frac 1 \beta} \leq \E[Z^{\beta}]^{\frac 1 \beta} = \E[X^{\alpha}]^{\frac 1 \alpha} \frac{\mathbb{E} [Z^{\beta}]^{\frac 1 \beta}}{\mathbb{E} [Z^{\alpha}]^{\frac 1 \alpha}} & \leq & \frac{1}{(1-p)^{\frac{1}{\alpha}}} \frac{\Gamma(\beta +1)^{1/\beta}}{\Gamma(\alpha+1)^{1/\alpha}} \E[X^{\alpha}]^{\frac 1 \alpha} \\ & \leq & \frac{\Gamma(\beta +1)^{1/\beta}}{\Gamma(\alpha+1)^{1/\alpha}} \E[X^{\alpha}]^{\frac 1 \alpha} e^{\frac{\Gamma(\alpha+1)^{\frac{1}{\alpha}}}{\alpha \E[X^{\alpha}]^{1/\alpha}}}.
\end{eqnarray}
\end{proof}

We note that Proposition \ref{moment-discrete} recovers the following well-known inequality in the continuous setting that can be traced back to Karlin-Proschan-Barlow \cite{karlin1961}, see also Bobkov-Madiman \cite[Corollary 3.2]{bobkov2011concentration}.

\begin{corollary}[Karlin-Proschan-Barlow \cite{karlin1961}]\label{moment-conti}
    Let $X$ be a non-negative continuous log-concave random variable, then
    \begin{align*}
         \E[X^{\beta}]^{\frac 1 \beta} \leq \frac{\Gamma(\beta +1)^{1/\beta}}{\Gamma(\alpha+1)^{1/\alpha}} \E[X^{\alpha}]^{\frac 1 \alpha}. 
    \end{align*}
\end{corollary}

\begin{proof}
    If $\mathbb{E}[X] = 0$, then $X=0$ and the proof is complete. Consider $\mathbb{E}[X] > 0$, and denote by $f$ the density of $X$. For $\varepsilon > 0$ define the log-concave probability sequence $f_\varepsilon \colon \mathbb{N} \to \mathbb{R}$ by $f_\varepsilon(n) = \frac{f(\varepsilon n)}{ \sum_m f(\varepsilon m)}$ and let $X_\varepsilon$ be a random variable with such probability mass function. With the definition
    $$\psi_p(\varepsilon) = \frac{ \mathbb{E}[X^p]^{\frac 1 p}}{\varepsilon \mathbb{E}[X_\varepsilon^p]^{\frac 1 p}},$$ we have
    \begin{align*}
        \mathbb{E} [X^\beta]^{\frac 1 \beta} 
            = 
                \varepsilon \psi_\beta(\varepsilon) \mathbb{E} [ X_\varepsilon^\beta]^{\frac 1 \beta}
            &\leq 
                \varepsilon \psi_\beta(\varepsilon)\frac{\Gamma(\beta + 1)^{1/\beta}}{\Gamma(\alpha + 1)^{1/ \alpha}} \mathbb{E}[X_\varepsilon^\alpha]^{\frac 1 \alpha} \exp \left[{\frac{\Gamma(\alpha+1)^{\frac{1}{\alpha}}}{\alpha \, \mathbb{E}[X_\varepsilon^{\alpha}]^{\frac{1}{\alpha}}}} \right]
                    \\
            &=
             \frac{\psi_\beta(\varepsilon)\Gamma(\beta + 1)^{1/\beta}}{\psi_\alpha (\varepsilon)\Gamma(\alpha + 1)^{1/ \alpha}} \mathbb{E}[X^\alpha]^{\frac 1 \alpha} \exp \left[\frac{\Gamma(\alpha+1)^{\frac{1}{\alpha}}}{\alpha}  \frac{\varepsilon \psi_{\alpha}(\varepsilon)}{\E[X^{\alpha}]^{\frac 1 \alpha}}\right].
    \end{align*}
    Taking $\varepsilon \to 0$ completes the proof as $\psi_p(\varepsilon) \to 1$ for any $p > 0$, since $\sum_n (n \varepsilon)^p f(n\varepsilon) \varepsilon$ is a Riemann sum approximation of $\int_0^\infty x^p f(x) dx$.
\end{proof}

\begin{remark}

Corollary \ref{moment-conti} can be derived directly by majorization as well (Theorem \ref{mom-compar}). It remains to note that
$$ A_{\alpha, \beta} = \frac{\E[Z^{\beta}]^{\frac{1}{\beta}}}{\E[Z^{\alpha}]^{\frac{1}{\alpha}}} = \frac{\Gamma(\beta+1)^{\frac{1}{\beta}}}{\Gamma(\alpha+1)^{\frac{1}{\alpha}}}, $$
when $Z$ is an exponential random variable.

\end{remark}

The next result complements Proposition \ref{moment-discrete}.

\begin{proposition}\label{cor: discrete logconcave moments}

Let $X$ be a discrete log-concave random variable on $\N$. Let $0 < \alpha \leq 1$. Then, for all $\beta \geq \alpha$, we have
$$ \E[X^{\beta}]^{\frac 1 \beta} \leq \frac{\Gamma(\beta +1)^{1/\beta}}{\Gamma(\alpha+1)^{1/\alpha}} (\E[X^{\alpha}]+1)^{\frac 1 \alpha}. $$

\end{proposition}

\begin{proof}
Assume that $\alpha \leq 1$ and let $\beta \geq \alpha$. Note that for all integer-valued random variables $X$ with p.m.f. $f$, we have
$$ \E[X^{\alpha}] = \sum_{k \geq 1} k^{\alpha} f(k) \leq   \sum_{k \geq 1} k f(k) = \E[X]. $$
Therefore, the constraint $\E[X^{\alpha}] = \E[Z^{\alpha}]$, for $Z$ geometric on $\N$, implies
$$ \E[X^{\alpha}] \leq \E[Z] = \frac{1-p}{p}, $$
and thus
$$ \frac{1}{1-p} \leq 1+\frac{1}{\E[X^{\alpha}]}. $$
Since the constraint $\E[X^{\alpha}] = \E[Z^{\alpha}]$ also implies
$$ \E[X^{\beta}] \leq \E[Z^{\beta}], $$
we have
\begin{eqnarray*}
\E[X^{\beta}]^{\frac 1 \beta} \leq \frac{1}{(1-p)^{\frac{1}{\alpha}}} \frac{\Gamma(\beta +1)^{1/\beta}}{\Gamma(\alpha+1)^{1/\alpha}} \E[X^{\alpha}]^{\frac 1 \alpha} \leq \frac{\Gamma(\beta +1)^{1/\beta}}{\Gamma(\alpha+1)^{1/\alpha}} (\E[X^{\alpha}]+1)^{\frac 1 \alpha}.    
\end{eqnarray*}
\end{proof}





For integer moments we derive sharp bounds for discrete log-concave random variables. One of several results in this direction is the following theorem of Keilson \cite{keilson1972threshold}, which we improve upon in the ultra log-concave case (see Section \ref{sec: ULC}). These can be considered discrete analogs of the well known fact that in the continuous setting, the function $p \mapsto \frac{1}{\Gamma(p+1)} \int_0^{+\infty} t^p f(t) dt$ is log-concave whenever $f$ is log-concave. This also complements the recent confirmation in \cite{melbourne2021discrete} of the conjectured log-concavity of the map $p \mapsto p \sum_k x_k^p$ for $x_k$ a monotone log-concave sequence (see \cite{MT20}). We use the notation $(m)_n = \frac{m!}{(m-n)!}$, with the convention that $(m)_n = 0$ if $m<n$.

\begin{theorem}[Keilson \cite{keilson1972threshold}]\label{thm: factorial moments}

Let $X$ be a discrete log-concave random variable supported on $\N$. Then the function
$$ \Phi \colon p \mapsto \frac{1}{p!} \E[(X)_p] $$
is log-concave on $\N$, where $\E[(X)_p]$ is the $p$-th factorial moment of $X$.

\end{theorem}

We provide a new proof of Theorem \ref{thm: factorial moments} based on convex majorization.

\begin{proof}[Proof of Theorem \ref{thm: factorial moments}]
Denote by $f$ the probability mass function of $X$. Denote, for $p \in \N$, $\Phi(p) = \frac{1}{p!} \E[(X)_p] \geq 0$. We want to prove that for all $p \in \N$,
$$ \Phi(p+1) \geq \sqrt{\Phi(p) \Phi(p+2)}. $$
For this, fix $p \in \N$. Assume $\Phi(p+2) > 0$, otherwise the result is trivial. Therefore $\E[(X)_{p+2}] > 0$, which implies that there exists $k_0 \geq p+2$ such that $f(k_0) > 0$, and thus $\Phi(p), \Phi(p+1) > 0$. Consider the log-affine function $g(k) = ca^k$, $k \in \N$, where $c$ and $a$ are chosen such that
$$ a = \frac{\Phi(p+1)}{\Phi(p) + \Phi(p+1)} \in (0,1) \quad \mbox{and} \quad c = \left( \frac{\Phi(p)}{\Phi(p+1)} \right)^{p+1} \frac{\Phi(p) \Phi(p+1)}{\Phi(p) + \Phi(p+1)}. $$
Such a choice ensures that
\begin{equation}\label{eq-kleitman}
\sum_{k \geq 0} (k)_p ca^k = \E[(X)_p] \quad \mbox{and} \quad \sum_{k \geq 0} (k)_{p+1} ca^k = \E[(X)_{p+1}],
\end{equation}
since for all $l \in \N$,
$$ \sum_{k \geq 0} (k)_l ca^k = c a^l \sum_{k \geq l} k \cdots (k-l+1) a^{k-l} = c a^l \frac{l!}{(1-a)^{l+1}}, $$
where the last identity follows by taking the $l$-th derivative of $\sum_{k \geq 0} a^k = \frac{1}{1-a}$. 
In particular,
\begin{equation}\label{p+2}
\sum_{k \geq 0} (k)_{p+2} \, c a^k = (p+2)! \, \frac{c}{1-a} \left( \frac{a}{{1-a}} \right)^{p+2} = (p+2)! \, \frac{\Phi(p+1)^2}{\Phi(p)}.
\end{equation}
Define
$$ \widetilde{f} = \frac{(k)_p f}{\E[(X)_p]} \quad \mbox{and} \quad \widetilde{g} = \frac{(k)_p g}{\E[(X)_p]}. $$
Clearly $\widetilde{f}$ defines a distribution and $\widetilde{g}$ as well by \eqref{eq-kleitman}. Denote $\widetilde{X}$ (resp. $\widetilde{Z}$) the random variable with distribution $\widetilde{f}$ (resp. $\widetilde{g}$), and note that $\widetilde{X} \prec_{lc} \widetilde{Z}$ since $f$ is log-concave and $g$ is log-affine. Define the non-decreasing function $T(k) = \frac{(k)_{p+1}}{(k)_p} = (k-p)1_{\{k \geq p+1\}}$, and note that
$$ \E[T(\widetilde{X})] = \frac{\E[(X)_{p+1}]}{\E[(X)_{p}]} = \E[T(\widetilde{Z})], $$
where the last equality follows by \eqref{eq-kleitman}. Therefore, taking $\mu$ to be the distribution of $\widetilde{Z}$, one may apply Theorem \ref{thm: extension of convex domination theorem} to deduce that $T(\widetilde{X}) \prec_{cx} T(\widetilde{Z})$, and thus for any convex function $\varphi$,
\begin{equation}\label{eq-maj-kleitman}
\E[\varphi(T(\widetilde{X}))] \leq \E[\varphi(T(\widetilde{Z}))].
\end{equation}
Choosing the convex function $\varphi(k) = k(k-1)$ yields 
$$ \varphi(T(k)) = (k-(p+1))(k-p)1_{\{k \geq p+2\}} = \frac{(k)_{p+2}}{(k)_p}, $$
hence inequality \eqref{eq-maj-kleitman} yields 
$$ \E[(X)_{p+2}] \leq \sum_{k \geq 0} (k)_{p+2} g(k) = (p+2)! \, \frac{\Phi(p+1)^2}{\Phi(p)}, $$
where the last identity follows from \eqref{p+2}. This gives the desired result.
\end{proof}

\begin{remark}
    
One easily deduces from Theorem \ref{thm: factorial moments} the following sharp moment inequalities
$$ \left(\frac{\E[(X)_{r+1}]}{(r+1)!}\right)^{\frac{1}{r+1}} \leq \left(\frac{\E[(X)_r]}{r!}\right)^{\frac{1}{r}}, \quad r \in \N, $$
holding for all discrete log-concave random variables.

\end{remark}

We end this section with entropy bounds for discrete log-concave random variables.

\begin{theorem}\label{thm: max-entropy-lc}
    For $\alpha \leq 1$ and $X$ a log-concave distribution on $\mathbb{N}$ with $\mathbb{E}[X] = \mu$
    \[
        H_\alpha(X) \leq H_\alpha(Z) = \frac{\log \left( (1 + \mu)^\alpha - \mu^\alpha \right)}{\alpha - 1} < \log  (1  +\mu)  + \frac{\log \alpha }{ \alpha - 1}
    \]
   where $Z$ is a geometric distribution on $\mathbb{N}$ with parameter $p = \frac{\mu}{1 +\mu}$.
\end{theorem}

\begin{proof}
    The proof follows immediately from Lemma \ref{entropy} and Theorem \ref{thm: maximum entropy dist}. The strict inequality follows from $(1 + \mu)^\alpha - \mu^\alpha > \alpha (1+\mu)^{\alpha - 1}$ by the Lagrange mean value theorem.
\end{proof}

Note that the geometric distribution maximizes the Shannon entropy among {\it all} $\mathbb{N}$-valued random variables with fixed mean, so the above theorem is only interesting for $\alpha < 1$. When $\alpha > 1$, the situation is more intricate. In fact, it turns out that Theorem \ref{thm: max-entropy-lc} does not hold for $\alpha > 1$. Indeed, let $Z$ be a geometric distribution on $\N$ with parameter $p \in (\frac{1}{2}, 1)$ and let $X$ be a Bernoulli distribution with parameter $\frac{1-p}{p}$, so that $\E[X] = \E[Z]$. We have
$$ e^{(1-\alpha) H_{\alpha}(X)} = \left( \frac{1-p}{p} \right)^{\alpha} + \left( \frac{2p-1}{p} \right)^{\alpha}, $$
$$ e^{(1-\alpha) H_{\alpha}(Z)} = \frac{p^{\alpha}}{1-(1-p)^{\alpha}}. $$
We claim that as $p \to 1$,
$$ \left( \frac{1-p}{p} \right)^{\alpha} + \left( \frac{2p-1}{p} \right)^{\alpha} < \frac{p^{\alpha}}{1-(1-p)^{\alpha}}, $$
and therefore $H_{\alpha}(X) > H_{\alpha}(Z)$. Indeed, the above inequality is equivalent to
$$ p^{2 \alpha} - [(1-p)^{\alpha} + (2p-1)^{\alpha}](1-(1-p)^{\alpha}) > 0. $$
L'H\^opital's rule can be used to show that
$$ \lim_{p \to 1} \frac{p^{2 \alpha} - [(1-p)^{\alpha} + (2p-1)^{\alpha}](1-(1-p)^{\alpha})}{(p-1)^2} = \alpha, $$
and the result follows. 

Similarly, comparisons of matched mean Bernoulli, Poisson, and Binomial, show that for small enough fixed means, neither the Poisson in the class of ultra log-concave distributions, nor the Binomial in the class of ultra log-concave distributions of order $n$, are maximizers of the $\alpha$-R\'enyi entropy for $\alpha > 1$, which we summarize in the following remark.

\begin{remark}
Let $\alpha > 1$. There exists $\mu > 0$ and a Bernoulli random variable $X$ such that
$$ H_{\alpha} (X) >  \max\{H_{\alpha}(Z_1),H_\alpha(Z_2), H_\alpha(Z_3)\} $$
where $Z_1$ is a geometric, $Z_2$ a Poisson, and $Z_3$ a Binomial satisfying $\E[Z_i] = \E[X] = \mu$.
\end{remark}

\section{Ultra log-concavity of order $n$ and ULC}\label{sec: ULC}

\begin{theorem} \label{thm: ULC(n) concentration}
    Let $X$ be a ULC$(n)$ random variable. Let $p = \mathbb{E}[X]/n$. Then for $t \geq 0$,
    \begin{align*}
        \mathbb{P}(X \geq (p+t)n ) &\leq e^{-n D( p + t || p)},
        \\
        \mathbb{P}(X \leq (p - t)n) &\leq e^{-n D(p - t || p)}. 
    \end{align*}
    Here, $D(\alpha || \beta) \coloneqq \alpha \log \frac \alpha \beta + (1 - \alpha) \log \frac{1 - \alpha}{1 - \beta}$.
\end{theorem}

\begin{proof}
        Let $X$ be ULC$(n)$. Then by Theorem \ref{Major-log}, $X \prec_{cx} Z$ where $Z$ is Binomial$(n,p )$ with $p = \mathbb{E}[X]/n$.  By Theorem \ref{thm: Chernoff majorization},
        \[
            \mathbb{P}(X \geq (p+t) n) 
            \leq \inf_{\lambda > 0} \left( (1-p)e^{- \lambda (p+t) } + pe^{-\lambda (p+t-1)} \right)^n.
        \]
        Evaluating the right side at its minimizer, that is when $e^\lambda =  \frac{(1-p)(p+t)}{(1-p-t) p}$ if $p+t < 1$ and when $\lambda \to +\infty$ if $p+t \geq 1$, gives the result. The lower bound is derived similarly.
\end{proof}

Note that by Pinsker's inequality (see, e.g., \cite[Theorem 7.10]{polyanskiy2025information}) one has $D( p + t || p)$ and $D( p - t || p)$ larger than $2 t^2$, so that
\[
    \mathbb{P} \left( |X - \mathbb{E}[X] | \geq \lambda \sqrt{n} \right) \leq 2 e^{-2 \lambda^2},
\]
and one recovers more transparently Gaussian concentration of ultra log-concave variables of order $n$.

We also recover the recent concentration inequalities derived for ULC random variables in \cite{aravinda2021concentration} used to generalize and improve the concentration inequalities for intrinsic volume random variables in \cite{lotz2020concentration} (see also \cite{NS22}). See Section \ref{sec: intrinsic} for further comments on the concentration of intrinsic volumes.
\begin{theorem} \label{thm: ULC(infinity) concentration}
Let $X$ be an ultra log-concave random variable. Then, for all $t \geq 0$,
$$ \P(X-\E[X] \geq t) \leq e^{-\frac{t^2}{2(t + \E[X])}}, $$
and
$$ \P(X - \E[X] \leq - t) \leq e^{-\frac{t^2}{2 \E[X]}}. $$
\end{theorem}

\begin{proof}
Let $Z$ be a Poisson distribution with parameter $\lambda = \E[X]$. Choosing the convex function $x \mapsto e^{tx}$ for arbitrary $t \in \R$ yields a pointwise dominance of moment generating functions
$$ \E[e^{tX}] \leq \E[e^{tZ}] = e^{\E[X] (e^t - 1)}. $$
The result then follows from a standard application of Markov's inequality (see \cite{aravinda2021concentration} for the details).
\end{proof}

The next result provides a strengthening of Keilson's result in Theorem \ref{thm: factorial moments} under ultra log-concavity.

\begin{theorem} \label{thm: factorial moments ULC}
    For $X$ a ULC($n$) random variable,
        \[
            \Phi \colon p \mapsto \frac{\mathbb{E}[(X)_p]}{(n)_p}
        \]
        is log-concave, or equivalently
        \begin{align} \label{eq: logconcavity of factorial moments}
            \mathbb{E}[ (X)_p]^2 \geq c_n(p) \mathbb{E}[(X)_{p+1}] \ \mathbb{E}[ (X)_{p-1}]
        \end{align}
        with $c_n(p) = 1 + \frac{ 1 }{n - p}$.  When $X$ is ULC, this is interpreted as $p \mapsto \mathbb{E}[(X)_p]$ is log-concave, or that \eqref{eq: logconcavity of factorial moments} holds with $c_\infty(p) = 1$ for all $p$.
\end{theorem}


\begin{proof}
    The proof  of Theorem \ref{thm: factorial moments} can be repeated, one needs only to utilize that $Z$ binomial$(n,q)$ satisfies
    \[
        \mathbb{E}[(Z)_l] = (n)_r q^l,
    \]
    for the ULC$(n)$ case; while $Z$ Poisson($\lambda$) satisfies
    \[
        \mathbb{E}[(Z)_l] = \lambda^l,
    \]
    for the ULC case.
\end{proof}


\begin{remark}

We easily deduce from Theorem \ref{thm: factorial moments ULC} the following sharp moment bounds for ULC$(n)$ random variables,
$$ \left(\frac{\E[(X)_{r+1}]}{(n)_{r+1}}\right)^{\frac{1}{r+1}} \leq \left(\frac{\E[(X)_r]}{(n)_r}\right)^{\frac{1}{r}}, \quad 0 \leq r \leq n-1. $$
In particular, ultra log-concave random variables satisfy $\E[(X)_{r+1}]^{\frac{1}{r+1}} \leq \E[(X)_r]^{\frac{1}{r}}$, for $r \in \N$.

\end{remark}

We end this section with a sharp entropy bound for ULC$(n)$ random variables, which is an immediate application of Lemma \ref{entropy} and Theorem \ref{thm: maximum entropy dist}.

\begin{theorem}
    Let $\alpha \leq 1$ and let $X$ be a ULC$(n)$ (resp. ULC) random variable. Then,
    \[
        H_\alpha(X) \leq H_\alpha(Z)
    \]
   where $Z$ is a binomial (resp. Poisson) distribution with $\E[Z] = \E[X]$.
\end{theorem}

\section{Log-concavity of intrinsic volumes}\label{sec: intrinsic}

Together with Aravinda in \cite{aravinda2021concentration}, the authors proved the pointwise domination of the moment generating function of ultra log-concave random variables with a fixed mean to derive concentration inequalities (using significantly more complicated techniques consisting in the identification of extreme points satisfying a linear constraint, see \cite{marsiglietti2020geometric}).  Using the fact that all intrinsic volume random variables are log-concave with respect to the Poisson of the same mean, one can conclude that intrinsic volumes exhibit at least ``Poisson-type Concentration''. The approach here immediately and easily implies all the results of \cite{aravinda2021concentration} as well as \cite{lotz2020concentration} as special cases.

Given convex bodies $K_1, K_2, \dots, K_n$ in $\mathbb{R}^d$ the function $f: [0,\infty)^n \to [0,\infty)$ defined by
\[
    f(t) = \vol_d(t_1 K_1 + \cdots + t_n K_n)
\]
is a $d$-homogeneous polynomial, and as such can be written as
\[
    f(t) = \sum_{i_1, \dots, i_d =1}^n V(K_{i_1}, \dots, K_{i_d} ) t_{i_1} \cdots t_{i_d}
\]
with coefficients $V(K_{i_1}, \dots, K_{i_d})$ that are symmetric under permutation of $(i_1, \dots, i_d)$.  The coefficient $V(K_{i_1}, \dots, K_{i_d})$ is the mixed volume of $K_{i_1}, \dots, K_{i_d}$.  For example when $K_i = K$, by the homogeneity of the Lebesgue measure, we have
\[
    \vol_d(t K) = t^d V(K,\dots, K)
\]
so that $V(K, \dots, K) = \vol_d(K)$.
\begin{theorem}[Alexandrov-Fenchel]
    For $d$-dimensional convex bodies $K_1, \dots, K_d$ we have the following inequality for mixed volumes
    \[
        V^2(K_1, K_2, K_3, \dots, K_d) \geq V(K_1, K_1, K_3, \dots, K_d) V(K_2, K_2, K_3, \dots, K_d).
    \]
\end{theorem}
Denote by $B_{j}$ the $j$-dimensional Euclidean unit ball and letting $\omega_{j} = \vol_{j} (B_j)$, and $B$ the $d$-dimensional unit ball.  Associate to a single $d$-dimensional convex body $K$ the sequence of {\it intrinsic volumes} given by $\{V_0(K), \dots, V_d(K) \}$ where $V_j(K)$ are the coefficients of the polynomial
\[
    \vol_d(K + t B_d) = \sum_{j=0}^d V_j(K) t^{d-j} \omega_{d-j}.
\]
Comparing with the expression of $\vol_d(K+tB)$ in terms of mixed volumes with $t = t_2$, $K_1 = K$ and $K_2 = B = B_d$ we can write $\vol_d(K + t B)$ as,
\begin{align*}
     \vol_d(t_1 K_1 + t_2 K_2) \bigg|_{t_1 = 1} 
        &= 
            \sum_{i_1, \dots, i_d = 1}^2 V(K_{i_1}, \dots, K_{i_d}) t_{i_1} \dots t_{i_d}
                \\
        &= 
            \sum_{j=0}^d \binom{d}{j}V(\underbrace{K, \dots, K}_{j \ { times}}, \overbrace{B, \dots, B}^{n-j \ { times}}) t^{d-j}.
\end{align*}
Writing $V(K^{(j)}, B^{(n-j)})$ for $V(\underbrace{K, \dots, K}_{j \ { times}}, \overbrace{B, \dots, B}^{n-j \ { times}})$ and equating the coefficients
\[
    V_j(K) = \frac{ \binom{d}{j}}{\omega_{d-j}} V(K^{(j)}, B^{(n-j)}).
\]
Note that the subscript $d$ is omitted in the expression of $V_j(K)$.  Intrinsic volumes of a convex body are independent of the ambient space in which they are embedded, see \cite{mcmullen1991inequalities, schneider2014convex}.
Note that by Alexandrov-Fenchel $j \mapsto V(K^{(j)}, B^{(n-j)}) $ is log-concave, which will be used to derive the following observation.
\begin{corollary} \label{cor: intrinsic volumes lc wrt ball}
    Given a $d$-dimensional convex body $K$, and $B$ the $d$-dimensional unit ball, the sequence
    \[
        j \mapsto \frac{V_j(K)}{V_j(\lambda B)} 
    \]
    is log-concave for $\lambda > 0$.  That is every intrinsic volume sequence is log-concave with respect to the intrinsic volume sequence of a ball.  In particular 
     \begin{equation} \label{eq: quantitative concavity of intrinsic volumes}
     {V^2_j(K)} \geq  \sqrt{ 1 - \frac{1}{d - j}} \ \left(1 + \frac{1}{j} \right) {V_{j-1}(K) V_{j+1}(K)}, \quad j=1, \dots, d-1.
 \end{equation}
    Thus the intrinsic volume sequence of a convex body is strictly ultra log-concave.  Moreover, 
    \begin{equation} \label{eq: convexity for the ball}
        V^2_j(B) \leq \sqrt{ 1 + \frac{1}{d - j - 1}} \ \left(1 + \frac{1}{j} \right) {V_{j-1}(B) V_{j+1}(B)}.
    \end{equation}
\end{corollary}
\begin{proof}
    Note that again comparing coefficients of the homogeneous polynomial $f(t)$ yields $V(\lambda K_1, K_2, \dots, K_n) = \lambda V(K_1, K_2, \dots, K_n)$ which by symmetry gives $V_j(\lambda B) = \lambda^j V_j(B)$. Since $j \mapsto \lambda^j$ is log-affine it suffices to consider the case that $\lambda = 1$.  We have
    \[
        \frac{V_j(K)}{V_j(B)} = \frac{ \frac{ \binom{d}{j}}{\omega_{d-j}} V(K^{(j)}, B^{(n-j)})}{  \frac{ \binom{d}{j}}{\omega_{d-j}} V( B^{(j)}, B^{(n-j)})}  =  \frac{ V(K^{(j)}, B^{(n-j)})}{\omega_d },
    \]
    which is log-concave by the Alexandrov-Fenchel inequality.
    To prove \eqref{eq: quantitative concavity of intrinsic volumes}, we follow the computations of McMullen \cite{mcmullen1991inequalities},
    \[
        \frac{V^2_j(\lambda B)}{V_{j-1}(\lambda B) V_{j+1}(\lambda B)} = \alpha_{d,j} \coloneqq \frac{j+1}{j} \beta_{d-j}
    \]
    where, after use of the identity $x\Gamma(x) = \Gamma(1 + x)$,
    \[
        \beta_s \coloneqq \frac{ \Gamma(s/2) \Gamma(1 + s/2)}{\Gamma^2((1+s)/2)}.
    \]
    Note that by the log-concavity of the first sequence we have $\frac{V^2_j(K)}{V_{j-1}(K) V_{j+1}(K)} \geq \alpha_{d,j}$, and in particular for $j < d$, by considering $B_d \subset \R^{d+1}$, we have $\alpha_{d+1,j} 
 \leq \frac{V^2_j(B_{d})}{V_{j-1}(B_{d}) V_{j+1}(B_{d})} = \alpha_{d,j} $.  Thus $\beta_s$ is decreasing in $s$.  Thus $\beta_s^2 \geq \beta_s \beta_{s+1} = \frac{s+1}{s} \geq \beta_{s+1}^2$.  Thus, we have for Euclidean Balls, and hence for any convex body $K \subseteq \mathbb{R}^d$,
 \[
     {V^2_j(K)} \geq  \sqrt{ 1 - \frac{1}{d - j}} \ \left(1 + \frac{1}{j} \right) {V_{j-1}(K) V_{j+1}(K)}.
 \]
The inequality $\sqrt{1 + \frac{1}{s-1}} \geq \beta_s$ coupled with  $\frac{V^2_j(\lambda B)}{V_{j-1}(\lambda B) V_{j+1}(\lambda B)} = \frac{j+1}{j} \beta_{d-j}$ yields \eqref{eq: convexity for the ball}.
\end{proof}

For a convex body $K$, we define its intrinsic volume random variable $X_K$ to be one given by the distribution
\[
    \mathbb{P}( X_K = j ) = \tilde{V}_{K}(j) \coloneqq \frac{V_j(K)}{\sum_{i=0}^d V_i(K)}.
\]
The central intrinsic volume of $K$ is defined to be
\[
    \Delta(K) = \mathbb{E}[X_K] = \sum_{j=0}^d {j \ \tilde{V}_K(j)}.
\]

\begin{proposition} \label{prop: central intrinsic volume increases on scaling}
    For a convex body $K$ the central intrinsic volume increases from $0$ to $d$ with scaling, that is the function 
        $$\lambda \mapsto \Delta(\lambda K)$$
        is strictly increasing for $\lambda \geq 0$, $\lim_{\lambda \to 0} \Delta(\lambda K) = 0$ and $\lim_{\lambda \to \infty} \Delta(\lambda K) = d$.
\end{proposition}

\begin{proof}
    By homogeneity, $\tilde{V}_{\lambda K}(j) = C(\lambda) \lambda^j \ V(K)$ for $C(\lambda) \coloneqq \left( \sum_{i=0}^d V_i(\lambda K) \right)^{-1}$.  Thus for $\lambda' > \lambda$ the function
    \[
        j \mapsto \frac{\tilde{V}_{\lambda' K}(j)}{\tilde{V}_{\lambda K}(j)} = C \left(\frac{\lambda'}{\lambda} \right)^j,
    \]
    with $C = C(\lambda')/C(\lambda)$, is strictly increasing.  Since by definition $\sum_j \tilde{V}_{\lambda K}(j) = \sum_j \tilde{V}_{\lambda' K}(j) =1$ it follows that there exits $k < d$ such that $\tilde{V}_{\lambda' K}(j) \leq \tilde{V}_{\lambda K}(j)$ for $j \leq k$ and $\tilde{V}_{\lambda' K}(j) >\tilde{V}_{\lambda K}(j)$ otherwise. Thus for $\Phi: \llbracket 0, d \rrbracket \to \mathbb{R}$ strictly increasing, 
    \begin{align*}
        \mathbb{E}[\Phi(X_{\lambda' K})] - \mathbb{E}[\Phi(X_{\lambda K})]
            =
                \sum_{j=0}^d (\Phi(j) - \Phi(k)) (\tilde{V}_{\lambda' K}(j) - \tilde{V}_{\lambda K}(j)) > 0,
    \end{align*}
    as the terms in the summands $\Phi(j) - \Phi(k)$ and $\tilde{V}_{\lambda'K}(j) - \tilde{V}_{\lambda K}(j)$ always have the same sign.  The case that $\Phi(x) = x$ gives our first result.  The limiting arguments are an easy exercise. 
\end{proof}

\begin{definition}[Intrinsic Entropy]
    For $\alpha \in [0,\infty]$ and a convex body $K$, define the intrinsic entropy of order $\alpha$ of $K$ as
    \[
        \hbox{IntEnt}_\alpha(K) \coloneqq H_\alpha(X_K)
    \]
\end{definition} 
That is the intrinsic entropy of order $\alpha$ associated to a convex body is just the $\alpha$-R\'enyi entropy of $X_K$.  The case that $\alpha = 1$ was studied in \cite{lotz2020concentration}.
\begin{theorem}\label{thm:intrinsic-entropy}
    For $\alpha \leq 1$, among convex bodies $K$ in $\mathbb{R}^d$ such that $\Delta(K) = \Delta$, the intrinsic entropy
    \begin{equation} \label{eq: intrinsic entropy maximizer}
        \hbox{IntEnt}_\alpha(K) = H_\alpha(X_K) = \frac{\log \left( \mathbb{E} {\tilde{V}_K^{\alpha -1}(X_K)} \right)}{1-\alpha}
    \end{equation}
    is maximized when $K$ is a $d$-dimensional Euclidean ball.  The optimal inequality independent of dimension is
    \begin{equation} \label{eq: Poisson maximizes entropy}
        \hbox{IntEnt}_\alpha (K) < H_{\alpha}(Z),
    \end{equation}
    where $Z \sim $ Poisson$(\Delta(K))$. 
\end{theorem}
We note that this corrects \cite[Theorem 6.1.13]{lotz2020concentration} where it is claimed that the cube, whose intrinsic volume random variables correspond to binomial distributions, maximizes intrinsic entropy (of order $\alpha =1$) for fixed expectation.  The error in the proof of \cite[Theorem 6.1.13]{lotz2020concentration} is the conflation of the spaces ULC($d$) and the ultra log-concave distributions with finite support on $[0,d]$ which we denote ULC$_{[d]}$($\infty$).  As mentioned in \cite{lotz2020concentration}, the binomial distribution does maximize entropy among ULC($d$) distributions with fixed expectation, see \cite{yu2008maximum}, however from \eqref{eq: convexity for the ball} we see that while $V_j(B)$ is ULC($\infty$), it does not belong to ULC($d$).  In summary, all intrinsic volumes sequences are log-concave with respect to the intrinsic volume of a ball, which is in turn log-concave with respect to a Poisson, however the intrinsic volume sequence of a ball is not log-concave with respect to a binomial.
 We mention in passing that the entropy maximizers of ULC$_{[d]}$($\infty$) for fixed expectation are the truncations of Poisson distributions, as can be proved by similar methods to that below. 

\begin{proof}[Proof of Theorem \ref{thm:intrinsic-entropy}]
    Note that the existence of $\lambda$ such that $\Delta(\lambda B) = \Delta$ is given by  Proposition \ref{prop: central intrinsic volume increases on scaling}.  By Corollary \ref{cor: intrinsic volumes lc wrt ball} the sequence $\tilde{V}_K(j)$ is log-concave with respect to $\tilde{V}_{\lambda B}(j)$ as each is just a constant multiple of $V_j(K)$ and $V_j(\lambda B)$ respectively.  From log-concavity it follows that $ \{ \tilde{V}_K(j) \geq \tilde{V}_{\lambda B}(j) \}$ is an interval $[a,b]$, and by the argument of Proposition \ref{prop: central intrinsic volume increases on scaling} we must have $[a,b] \subset (0,d)$ else we would have contradiction on $\Delta(K) = \Delta(\lambda B)$.  It follows that $X_K \prec_{cx} X_{\lambda B}$ since the densities ``cross twice'' and have matching expectation.  For the $\alpha =1$ case, since $\tilde{V}_{\lambda B}$ is log-concave, $- \log \tilde{V}_{\lambda B}$ is convex,
    \[
        H(X_{\lambda B}) = - \mathbb{E} \log \tilde{V}_{\lambda B}(X_{\lambda B}) \geq - \mathbb{E} \log \tilde{V}_{\lambda B}(X_{K}) \geq H(X_K),
    \]
    where the last inequality is Gibbs' inequality.  The case $\alpha < 1$ follows similarly. From log-concavity and the monotonicity of $L^p$ norms for probability measures (since $\alpha -1 < 0$),
    \[
        V_{\lambda B} (j) \geq \sqrt{V_{\lambda B}(j-1) V_{\lambda B}(j+1)} \geq \left( \frac{V_{\lambda B}^{\alpha - 1}(j-1)}{2} + \frac{V_{\lambda B}^{\alpha -1}(j+1)}{2} \right)^{\frac{1}{\alpha -1}},
    \]
    which yields $V_{\lambda B}^{\alpha - 1}$ to be convex, so that $\mathbb{E}[\tilde{V}_{\lambda B}^{\alpha -1}(X_{\lambda B})] \geq \mathbb{E}[\tilde{V}_{\lambda B}^{\alpha-1}(X_K)]$ with Lemma \ref{entropy} completes the proof.  
    
    As every intrinsic volume sequence is log-concave with respect to the Poisson, \eqref{eq: Poisson maximizes entropy} follows similarly.  
     Equality can be obtained asymptotically in \eqref{eq: Poisson maximizes entropy}, as a Poisson distribution is the limit of $\tilde{V}_{\lambda_d B^d}$ with $d \to \infty$ for $\lambda_d$  chosen to fix $\Delta( \lambda_d B^d)$.
\end{proof}

\section{Continuous setting}\label{sec:continuous}

Let us recall the notion of log-concavity of order $p$.
 
\begin{definition}[Bobkov \cite{bobkov2010gaussian}] \label{def: logconcave order p}
    A non-negative random variable $X$ is log-concave of order $p > 0$, if it possesses a density function $f$ that admits a decomposition such that
    \[
        f(x) = x^{p-1} g(x)
    \]
    for a non-negative, log-concave function $g$.
\end{definition}

Strictly speaking, log-concavity of order $p$ was defined in \cite{bobkov2010gaussian} (see also \cite{bobkov2003spectral, bobkov2011concentration, saumard2014log}) only for  $p \geq 1$.  In this case, $x^{p-1}$ is log-concave, and  the log concavity of order $p$ is a stronger assumption than log-concavity. Such distributions arise as the distribution from the norm of a spherically symmetric log-concave vectors. More explicitly if $X$ is log-concave and spherically symmetric in $\mathbb{R}^d $ with $d \geq 2$, then $\|X\|_2$ has a density on $(0,\infty)$ that is log-concave of order $d-1$.

For our purposes it is convenient to extend the definition of log-concavity of order $p$ to $p \in (0,1)$, as this class corresponds exactly to the densities that are log-concave with respect to a Gamma distribution.
We recall that the Gamma$(p,\beta)$ distribution is to be given by the density
\[
    f(x) = \frac{\beta^p}{\Gamma(p)} x^{p-1} e^{-\beta x}, \quad x > 0.
\]
Thus with Definition \ref{def: logconcave order p} a random variable $X$ is log-concave of order $p$ exactly when it is log-concave with respect to a Gamma$(p, \beta)$ distribution.  Notice that the statement is independent of the choice of $\beta$ and that the case $p =1$ corresponds to ordinary log-concavity. Recall also that for $Z \sim$ Gamma$(p, \beta)$, and $q  > -p$,
\[
    \mathbb{E} [Z^q] = \frac{\Gamma(p + q)}{ \Gamma(p) \beta^p}.
\]

The next result provides a bound on the moment generating function of random variables that are log-concave of order $p > 0$.

\begin{theorem}\label{mgf}
    For $X$ log-concave of order $p > 0$,  with mean $\mu = \mathbb{E}[X]$, $Z \sim$ Gamma$(p, \frac p \mu)$, and $t < \frac{p}{\mu}$
    \[
        \mathbb{E}[e^{tX}] \leq \mathbb{E}[e^{tZ}] = \left( 1 - \frac{\mu}{p}  t \right)^{-p} = \exp \left\{ -p \log \left( 1 - \frac {\mu t}{p} \right) \right\}.
    \]
\end{theorem}

Theorem \ref{mgf} extends and sharpens Bobkov's result \cite[Proposition 3.1]{bobkov2010gaussian} where it is proven that for $|t| < p/2 \mu$ and $p \geq 1$,
\[
    \mathbb{E}[e^{tX}] \leq \exp \left\{ \frac{t^2 \mu^2}{p} + t \mu  \right \}.
\]
 Through the standard application of Markov's inequality discussed in section \ref{sec: concentration}, we obtain the following sharpenings of the tail bounds in \cite{bobkov2010gaussian}.  
\begin{theorem}
    For $X$ log-concave of order $p>0$, $\mu \coloneqq \mathbb{E}[X]$, and $\lambda \geq 1$,
    \[
        \mathbb{P}( X \geq \lambda \mu) \leq  \exp \left\{ - p  \ ( \lambda - \log \lambda -1 \right) \}.
    \]
    For $\lambda \in (0,1)$,
    \[
    \mathbb{P}( X \leq \lambda \mu) \leq \exp \left\{ - p  \ ( \lambda - \log \lambda -1 \right) \}.
    \]
    Equivalently, for $h \geq 0$,
    \[
        \mathbb{P}( X - \mu \geq h \mu ) \leq \exp \{ - p ( h - \log( 1 + h) \}
    \]
    and with the constraint that $h \leq 1$
    \[
    \mathbb{P}( X - \mu \leq -h \mu ) \leq \exp \{ p ( h + \log( 1 - h) \}.
    \]
\end{theorem}

In particular when $p=1$, and similarly to the discrete setting, we extend the tail bounds for sums of independent exponential random variables due to Janson \cite{janson2018tail}, to sums of positive log-concave random variables.

\begin{theorem}\label{thm: conc exponential}

Let $X_1, \dots, X_n$, $n \geq 1$, be independent positive continuous log-concave random variables, with $\mathbb{E}[X_i] = 1$, and  $\lambda_i >0$ with $\lambda \coloneqq \sum_{i=1}^n \lambda_i$ and $S_n \coloneqq \sum_{i=1}^n \lambda_i X_i$.  Then, for all $t \geq 1$,
$$ \P(S_n \geq t \, \lambda) \leq e^{- \lambda \left(\min_i \lambda_i  \right) (t-1-\log(t))}, $$
and for all $t \leq 1$,
$$ \P(S_n \leq t \, \lambda) \leq e^{- \lambda \left(\min_i \lambda_i \right) (t-1-\log(t))}. $$

\end{theorem}

Note that taking $X_i$ to be i.i.d. exponential($1$), and $\lambda_i = \frac 1 {a_i}$ recovers \cite[Theorem 5.1 (i)]{janson2018tail} and \cite[Theorem 5.1 (iii)]{janson2018tail}.


We end this section with sharp moment and entropy bounds.

\begin{theorem}
    For $X$ log-concave of order $p >0$, and $r > q > - p$,
    \[
        \mathbb{E} [X^{r}]^{\frac 1 r} \leq C(p,q,r) \mathbb{E} [X^{q}]^{\frac 1 q}, 
    \]
    where
    $
        C(p,q,r) = \frac{\Gamma(p + r)^{\frac 1 r}}{\Gamma(p + q)^{\frac 1 q}} \Gamma(p)^{\frac 1 q - \frac 1 r}.
    $
\end{theorem}


\begin{theorem}
    For $\alpha \in [0,1]$, among $X$ log-concave of order $p \geq 1$, with mean $\mathbb{E}[X] = \mu$, $Z \sim$ Gamma$(p, \frac p \mu )$ has maximum $\alpha$-R\'enyi entropy.  That is, when $\alpha \in [0,1)$,
    \[
        h_{\alpha}(X) \leq h_{\alpha}(Z) = \frac{1}{1-\alpha} \log \left( \frac{\Gamma(1 + \alpha (p-1))}{\Gamma(p)^{\alpha} \beta^{1-\alpha} \alpha^{1+p\alpha}} \right),
    \]
    and
    \[
        h(X) \leq h(Z) = p - \log \frac p \mu + \log \Gamma(p) + (1-p) \psi(p),
    \]
    where $\psi(p)$ denotes the digamma function given by $\psi(p) \coloneqq \frac{\Gamma'(p)}{\Gamma(p)}$.
\end{theorem}

We omit the proofs of the results in this section, as the proofs are based on convex majorization (Theorems \ref{thm: Chernoff majorization}, \ref{mom-compar} and \ref{thm: maximum entropy dist}), similarly as to the rest of the article.

\vskip5mm
\noindent
{\bf Acknowledgments.}
We thank Mokshay Madiman for pointing out to us the article of Yaming Yu \cite{Yu09:2}.  The second author extends his gratitude to Artem Zvavich for introducing him to the Alexandrov-Fenchel inequality, as well as to Arturo Jaramillo for fruitful discussions. Finally, we thank the anonymous referees for their helpful comments that improved the presentation of the article.

\bibliographystyle{plain}
\bibliography{bibibi}

\vspace*{1cm}

\noindent Arnaud Marsiglietti \\
Department of Mathematics \\
University of Florida \\
Gainesville, FL 32611, USA \\
E-mail: a.marsiglietti@ufl.edu

\vspace{0.8cm}

\noindent James Melbourne \\
Probabilidad y Estad\'isticas\\
Centro de Investigaci\'ones en Matem\'aticas \\
Guanajuato, GTO 36023, MX \\
E-mail: james.melbourne@cimat.mx

\end{document}